\documentclass[11pt]{amsart}
\usepackage{amsmath,amssymb,amscd,amsthm}
\usepackage{mathrsfs}

\numberwithin{equation}{section}

\theoremstyle{plain}
\newtheorem{thm}{Theorem}[section]
\newtheorem{theorem}[thm]{Theorem}
\newtheorem{lemma}[thm]{Lemma}

\theoremstyle{definition}

\newtheorem{remark}[thm]{Remark}
\newtheorem{definition}[thm]{Definition}

\newcommand{\smp}[1]{\left(#1 \right)}

\newcommand{\R}{\mathbb{R}}
\DeclareMathOperator{\Sym}{Sym}
\newcommand{\mn}{\sqrt{-1}}
\DeclareMathOperator{\trace}{tr}
\newcommand{\tr}[2]{\textrm{tr}_{#1}{#2}}
\newcommand{\ti}[1]{\tilde{#1}}

\newcommand{\E}{\mathcal{E}}

\providecommand{\abs}[1]{| #1|}
\providecommand{\norm}[1]{\lVert#1\rVert}

\renewcommand{\leq}{\leqslant}
\renewcommand{\geq}{\geqslant}
\renewcommand{\le}{\leqslant}
\renewcommand{\ge}{\geqslant}
\newcommand{\de}{\partial}
\newcommand{\db}{\overline{\partial}}

\newcommand{\ddb}{\partial \ov{\partial}}
\newcommand{\ddbar}{\sqrt{-1} \partial \overline{\partial}}
\newcommand{\ov}[1]{\overline{#1}}

\newcommand{\proj}{\mathrm{p}}

\title[Nonlinear elliptic equations on complex manifolds]{$C^{2,\alpha}$ estimates for  nonlinear elliptic equations in complex and almost complex geometry}
\author[V. Tosatti]{Valentino Tosatti}
\author[Y. Wang]{Yu Wang}
\author[B. Weinkove]{Ben Weinkove}
\author[X. Yang]{Xiaokui Yang}
\address{Department of Mathematics, Northwestern University, 2033 Sheridan Road, Evanston, IL 60208}
\thanks{Research supported in part by NSF grants DMS-1236969, DMS-1308988 and DMS-1332196.  The first named-author is supported in part by a Sloan Research Fellowship.}

\begin{document}
\begin{abstract} We describe how to use the perturbation theory of Caffarelli to prove Evans-Krylov type $C^{2,\alpha}$ estimates for solutions of  nonlinear elliptic equations  in complex  geometry, assuming a bound on the Laplacian of the solution.  Our results can be used to replace the various Evans-Krylov type arguments in the complex geometry literature with a sharper and more unified approach.  In addition, our methods extend to  almost-complex manifolds, and we use this to obtain a new local  estimate for an equation of Donaldson.
\end{abstract}

\maketitle

\section{Introduction}

The Evans-Krylov theorem \cite{Ev, Kr} for nonlinear elliptic equations
$$F(D^2 u )= f,$$
with $F$ concave,
has long been used in the study of PDEs to show that bounds on $u$ and $D^2u$ imply
$C^{2,\alpha}$ bounds on $u$ for some $\alpha>0$.  This theory applies whenever $f$ is suitably bounded and the  bounds on $u$ imply uniform ellipticity. This scenario appears in the study of many PDEs, including the real Monge-Amp\`ere equation, the $\sigma_k$-equations and their  differential-geometric counterparts.

At the same time, nonlinear concave operators occur frequently in the study of complex geometry.  Complex-differential techniques often yield the estimates
\begin{equation} \label{assumeu0}
\| u \|_{L^{\infty}} \le K, \qquad \Delta u \le K,
\end{equation}
and one would like to conclude that $u$ is bounded in $C^{2,\alpha}$.  In the examples, as we will see, the bound on the Laplacian of $u$ typically gives bounds for the \emph{complex} Hessian of $u$, but not on the \emph{real} Hessian of $u$.

Up until recently there have been two approaches to the $C^{2,\alpha}$ estimate in complex geometry.  The first is to carefully reprove the Evans-Krylov estimates in the complex setting by differentiating the equation twice in complex coordinates and using concavity together with a Harnack inequality.  The second approach is to establish estimates on the \emph{real} Hessian $D^2 u$, using bounds that one typically has on the complex Hessian, and then apply the real Evans-Krylov theory referred to above. Both approaches require considerable work to carry out, and are also suboptimal, in the sense that the $C^{2,\alpha}$ estimates depend on \emph{two derivatives} of $f$.

Recently the second-named author showed that \cite{Wa}, for the complex Monge-Amp\`ere equation,
$$\det \left( \frac{\partial^2 u}{\partial z^i \partial \ov{z}^j} \right)=e^{\psi},$$
one can instead directly apply the theory of Caffarelli \cite{Ca} to obtain local $C^{2,\alpha}$ estimates depending only on the constant $K$ of (\ref{assumeu0}) and a H\"older estimate of $\psi$.  This settled a regularity issue in complex geometry related to \cite{CT}  (see also \cite{DZZ}).

In this paper, we generalize the method of \cite{Wa} to a number of nonlinear equations which naturally occur in complex (and almost complex) geometry.  We show that for these equations, under the assumptions (\ref{assumeu0}), the $C^{2,\alpha}$ estimates follow from the general nonlinear PDE theory developed by Caffarelli.  This can replace the menagerie of arguments in the literature by a single unified theorem.   Moreover, we hope that our results can be used as a convenient ``black box'' in the future study of elliptic equations in complex geometry.

We should stress that the contributions of this paper are not to the general theory of real nonlinear elliptic PDE.  Rather our  aim is to demonstrate how the existing PDE theory (some of which may be unfamiliar to complex geometers) can be used to simplify and improve the $C^{2,\alpha}$ estimate for many examples of elliptic equations appearing in complex geometry.

We now describe the various equations in complex geometry for which our results can be applied.  Let $(M, J, \omega)$ be a compact Hermitian manifold.  Suppose that the real-valued function $u \in C^2(M)$ satisfies
\begin{equation} \label{assumeu}
\| u \|_{L^{\infty}(M)} \le K, \ \textrm{and } \ \Delta u \le K,
\end{equation}
where $$\Delta u = \frac{n\omega^{n-1} \wedge \ddbar u}{\omega^n}$$ is the complex Laplacian associated to $\omega$.  Note that in all of the cases described below, the $L^{\infty}$ bound on $u$ could be replaced by a normalization condition, say $\int_M u=0$, once the Laplacian bound is given.

Assume that $\psi \in C^{\alpha_0}(M)$  and $\chi$ is a real $(1,1)$ form with coefficients in $C^{\beta}(M)$.  We allow the possibility that $\chi$ and $\psi$ depend on $u$, and indeed this is important for our applications.  We consider $u$ satisfying one of the following equations.

\bigskip
\noindent
\emph{The complex Monge-Amp\`ere equation:}
\begin{equation} \label{eqnma}
\begin{split}
(\chi + \ddbar u)^n = {} & e^{\psi} \omega^n \\
 \quad \chi + \ddbar u> {} & 0.
\end{split}
\end{equation}

\noindent
\emph{The complex Hessian equations:  for a fixed $k=2, \ldots, n-1$,}
\begin{equation} \label{eqnch}
\begin{split}
(\chi + \ddbar u)^k \wedge \omega^{n-k} = {} &  e^{\psi} \omega^n, \\
(\chi + \ddbar u)^j\wedge \omega^{n-j} > {} & 0, \quad \textrm{for } j=1, \ldots, k.
\end{split}
\end{equation}

\noindent
\emph{The complex $\sigma_n/\sigma_k$ equations: for a fixed $k=2, \ldots, n-1$,}
\begin{equation} \label{eqnd}
\begin{split}
(\chi + \ddbar u)^k \wedge \omega^{n-k} = {} & e^{\psi} (\chi + \ddbar u)^n \\
\chi + \ddbar u > {} & 0.
\end{split}
\end{equation}

\noindent
\emph{The Monge-Amp\`ere equation for $(n-1)$-plurisubharmonic equations:}
\begin{equation} \label{eqnnm1}
\begin{split}
\left( \chi + \frac{1}{n-1} \left( (\Delta u) \omega - \ddbar u \right) \right)^n = {} & e^{\psi} \omega^n \\
\chi + \frac{1}{n-1} \left( (\Delta u) \omega - \ddbar u \right) >{} &  0.
\end{split}
\end{equation}

\noindent
\emph{The $(n-1)$-plurisubharmonic version of the complex Hessian and $\sigma_n/\sigma_k$ equations:}
\begin{equation} \label{eqnnm2}
\begin{array}{c}
\textrm{replace $\ddbar u$ by $\frac{1}{n-1}((\Delta u) \omega - \ddbar u)$} \\
\textrm{and suppose $u$ satisfies one of  (\ref{eqnch}) or (\ref{eqnd}).}
\end{array}
\end{equation}

\noindent
\emph{Almost complex versions of all of the above.} Replace $(M, J, \omega)$ Hermitian by  $(M, J, \omega)$ almost Hermitian, $\chi$ by a real $(1,1)$ form w.r.t. $J$ and  $\ddbar u$ by $\frac{1}{2}(dJdu)^{(1,1)}$ and suppose that
\begin{equation} \label{eqnac}
\begin{array}{c}
\textrm{with these replacements,} \\
\textrm{$u$ satisfies one of (\ref{eqnma}), (\ref{eqnch}), (\ref{eqnd}),  (\ref{eqnnm1}) or (\ref{eqnnm2})}.
\end{array}
\end{equation}

Our main result is:

\begin{theorem} \label{theoremCG}
Assume that $u\in C^2(M)$ satisfies (\ref{assumeu}) on $(M, J, \omega)$ and any one of (\ref{eqnma}), (\ref{eqnch}), (\ref{eqnd}), (\ref{eqnnm1}), (\ref{eqnnm2}) or (\ref{eqnac}).  Then
$$\| u \|_{C^{2,\alpha}(M)} \le C,$$
where $\alpha$, $C$ depend only on $n$, $\beta$, $(M, J, \omega)$, $K$, $\|\psi \|_{C^{\alpha_0}}$ and $\| \chi \|_{C^{\beta}}$.
\end{theorem}

A first remark is that although we have stated the results for \emph{compact} $M$, this is only for convenience.  All our estimates and results are purely local.

We now discuss each of the equations (\ref{eqnma}), (\ref{eqnch}), (\ref{eqnd}), (\ref{eqnnm1}), (\ref{eqnnm2}) and (\ref{eqnac}) in turn, describing briefly some of the existing results in the literature.

 We begin with the complex Monge-Amp\`ere equation (\ref{eqnma}) when $\chi$ is a fixed \emph{K\"ahler} metric (i.e. $\chi>0$ and $d\chi=0$).  Global existence of solutions was proved in the seminal work of Yau \cite{Ya}, in which he used
   Calabi's third order estimate \cite{Cal} to establish the $C^{2,\alpha}$ estimate.  This estimate depends on \emph{three} derivatives of $\psi$.   A direct proof depending on two derivatives of $\psi$, using the Evans-Krylov approach,  was given by Siu \cite{Si} (cf. Trudinger \cite{Tr}).  A $C^{2,\alpha}$ estimate depending only on the H\"older bound for $\psi$, but also on a  bound for the \emph{real} Hessian of $u$ was given by Dinew-Zhang-Zhang \cite{DZZ}.  As discussed above, the second-named author \cite{Wa}  established the result of Theorem \ref{theoremCG} in the case of the  complex Monge-Amp\`ere equation (under slightly weaker hypotheses for $u$).

In the general setting of a fixed Hermitian metric $\chi$,  the existence of solutions to the complex Monge-Amp\`ere equation (\ref{eqnma}) was established by Cherrier \cite{Ch} for $n=2$ (and in higher dimensions with additional hypotheses) and by the first and third-named authors in general \cite{TW2}.  Cherrier \cite{Ch} established a $C^{2,\alpha}$ bound via
 a Calabi-type third order estimate.  Guan-Li \cite{GL} proved a $C^{2,\alpha}$ estimate by first establishing a bound on the real Hessian of $u$ using a maximum principle argument and then applying the usual Evans-Krylov theory.  In \cite{TW1}, a direct Evans-Krylov argument was given, following a similar approach to that given in the notes of Siu \cite{Si}.  All of these existing estimates depend on at least two derivatives of $\psi$. Theorem \ref{theoremCG} can replace and improve these results.

The complex Hessian equations (\ref{eqnch}) with $\chi=\omega$ K\"ahler were solved by Dinew-Ko{\l}odziej \cite{DK}, who made use of estimates of Hou \cite{Hou} and Hou-Ma-Wu \cite{HMW}.  For related works see \cite{Bl, DK2, J, Li, Lu, N}.  A $C^{2,\alpha}$ estimate in this case was proved by Jbilou \cite{J} by bounding the real Hessian and then applying the real Evans-Krylov theory (see also \cite{Hou}, \cite{Kok}).

The $\sigma_n/\sigma_k$ equations for $k=n-1$ and $\chi$ a fixed K\"ahler metric was introduced by Donaldson \cite{Do} in the context of global complex geometry.  When $n=2$, Chen \cite{Chen1} observed that the equation reduces to the complex Monge-Amp\`ere equation (\ref{eqnma}).  Necessary and sufficient conditions for existence of solutions for $k=n-1$ were given by Song-Weinkove \cite{SW} using a parabolic method (see \cite{Chen2, We1, We2}).   For general $k$, the analogous result was proved by Fang-Lai-Ma \cite{FLM}, again using a parabolic equation.  The elliptic equation, for  $\chi$ a Hermitian metric, was solved by Sun \cite{Sun} (see also the work of Guan-Sun \cite{GS} and Li \cite{LiYi}).  In particular, Sun established the elliptic Evans-Krylov $C^{2,\alpha}$ estimate \cite{Sun}, again depending on two derivatives of $\psi$.

The Monge-Amp\`ere equation for $(n-1)$-plurisubharmonic functions (\ref{eqnnm1}) has appeared in the literature in various forms, corresponding to different choices of $\chi$ and $\omega$.   Harvey-Lawson introduced this equation and the notion of $(n-1)$-plurisubharmonic functions in $\mathbb{C}^n$ \cite{HL2}.  In the case of compact manifolds, the simplest case is when $\chi$ is a fixed Hermitian metric and $\omega$ is K\"ahler.  This was introduced and investigated by Fu-Wang-Wu \cite{FWW, FWW2}, motivated by some questions related to  mathematical physics (see \cite{LY, FY}, for example).
The equation (\ref{eqnnm1}) in this setting was solved by the first and third-named authors \cite{TW4}, and then more recently extended to the case of $\omega$ Hermitian \cite{TW5}.

Another important setup for equation (\ref{eqnnm1}) is the case
$$\chi=\chi_0 + *E,$$
where $\chi_0$ is a Hermitian metric, $*$ is the Hermitian Hodge star operator of $\omega$, and
$$E=\frac{1}{(n-1)!}\mathrm{Re}\left(\mn\de u\wedge\db(\omega^{n-2}) \right).$$
Note that Theorem \ref{Apply EK} still applies, because by the bounds \eqref{assumeu} we conclude that $\|\nabla u\|_{C^{\beta}(M)}\leq C$, for a uniform constant $C$, and hence $\chi$ has a uniform $C^\beta$ bound.
The equation (\ref{eqnnm1}) with this choice of $\chi$ was introduced by Popovici \cite{Po}, motivated by some questions in algebraic geometry, and was also studied by the first and third-named authors  \cite{TW5} as an approach towards a conjecture of Gauduchon \cite{Ga}.  The general question of existence of solutions remains open.

There is yet another natural choice of $\chi$ in (\ref{eqnnm1}):
$$\chi=\chi_0 + 2*E+u\mn\ddb(\omega^{n-2}),$$
where $\chi_0$ is a Hermitian metric, and $*E$ is the same as above.
In this case,
$$\chi+\frac{1}{n-1} \left(  (\Delta u) \omega - \ddbar u \right)=*\left(\omega_0^{n-1}+\ddbar(u\omega^{n-2})\right),$$
where $\omega_0$ is Hermitian and $\omega_0^{n-1}+\ddbar(u\omega^{n-2})$ is a positive $(n-1,n-1)$ form. This choice of $\chi$ gives an equation
introduced by Fu-Wang-Wu \cite{FWW}.
It is not known in general whether it can be solved (except when $\omega$ is K\"ahler  \cite{TW4}.)

The $C^{2,\alpha}$ estimates for (\ref{eqnnm1}) in the above settings (again, depending on two derivatives of $\psi$) were established in \cite{FWW, TW4, TW5} by adapting the usual Evans-Krylov approach.  Theorem \ref{theoremCG} replaces and sharpens these arguments.

As far as we know, the equations (\ref{eqnnm2}) have yet to be studied, at least in these  explicit forms on compact manifolds.  However, we have included these equations since they appear to be natural PDEs which easily fit into our setting.

The Dirichlet problem for the almost complex Monge-Amp\`ere equation (\ref{eqnac}) has been solved in various settings by Harvey-Lawson \cite{HL} and Pli\'s \cite{Pl}.  A related, but different, equation in the almost complex case was investigated by Delano\"e \cite{De}.     Donaldson's Calabi-Yau equation for almost K\"ahler forms \cite{Do2}, which is not ostensibly of the form (\ref{eqnac}), does in fact fit into this setting. Evans-Krylov results were proved for this equation in \cite{We3} and \cite{TWY}.  In Section \ref{sectiondon} below, we describe Donaldson's equation and how our results can be used to prove a new local Evans-Krylov estimate (Theorem \ref{theoremdon}).

We next describe the local PDE theorem which we use to prove Theorem \ref{theoremCG}.  This PDE theorem is a consequence of results of Caffarelli \cite{Ca} and ideas adapted from the work of the second-named author \cite{Wa} (see also \cite{Sa}).
We have packaged the theorem in a way to make it easy to apply to our examples.

Let $B_1$ be the unit ball in $\mathbb{R}^{2n}$.   Write $\Sym (2n )$ for the space of symmetric $2n \times 2n$ matrices with real entries.
We consider equations of the form
\begin{equation} \label{me}
   F\left( S(x)  + T \smp{D^2 u (x), x },x \right) = f(x), \quad \textrm{for } f \in C^{\alpha_0}, \ x \in B_1,
\end{equation}
where
\[
\begin{split}
& F : \Sym (2n )\times B_1 \rightarrow \R, \\
& S: B_1 \rightarrow \Sym (2n), \\
& T :\Sym (2n ) \times B_1 \rightarrow \Sym(2n).
\end{split}
\]
For any $A\in\Sym(2n)$ will denote by $\|A\|$ its operator norm, i.e. the maximum of the absolute value of the eigenvalues of $A$.
We impose  the following structure conditions on $F, S$ and  $T$.  We assume that there exists a compact convex set $\E \subset \Sym(2n)$, positive constants $\lambda, \Lambda, K$ and $\beta \in (0,1)$ such that the following hold.\\

\begin{description}
\item[H1]  $F$ is of class $C^1$ in $U\times B_1$ where $U$ is a neighborhood of $\E$ and
\begin{enumerate}
\item $F$ is uniformly elliptic in $\E$:
\[
\lambda | \xi |^2   \leq \sum_{i,j}F_{ij}(A,x) \xi^i \xi^j \le \Lambda | \xi|^2,
\]
for all $A\in\E, \xi \in \mathbb{R}^{2n}, x\in B_1$, where $F_{ij}(A,x) = \frac{\partial F}{\partial A_{ij}} (A,x)$.

\item $F$ is concave in $\E$:
\[
 F \smp{ \frac{A+ B}{2},x} \geq  \frac{1}{2} F(A,x)  + \frac{1}{2} F(B,x) ,\quad \textrm{for all} \ A, B \in \E, x\in B_1.
\]

\item $F$ has the following uniform H\"older bound in $x$:
$$|F(N,x)-F(N,y)|\leq K|x-y|^\beta, \quad \textrm{for all} \ N\in \E, x,y\in B_1,$$
and $|F(N,0)| \le K$, for all $N\in\E$.
\end{enumerate}

\item[H2] The map $T :  \Sym(2n) \times B_1 \rightarrow \Sym(2n)$ satisfies the following conditions:
\begin{enumerate}
\item For all $x, y \in B_1$ and all $N  \in \Sym(2n)$,
\[
\frac{\norm{T (N, x)  - T(N, y) }}{\norm{N}+1} \leq K |x -y|^\beta.
\]
\item For each fixed $x \in B_1$, the map $N \mapsto T(N, x )$ is  linear on $\Sym(2n)$.
\item For all $P \ge 0$ and $x \in B_1$,
\[
\begin{split}
T(P,x) \ge 0, \  \textrm{and} \ K^{-1} \| P \| \leq  {} &   \| T ( P,x) \|  \leq K \| P\|.
\end{split}
\]
\end{enumerate}

\item[H3] $S : B_1 \rightarrow \Sym(2n)$ has a uniform $C^{\beta}$ bound:
$$\| S(x) - S(y) \| \le K |x-y|^{\beta}, \quad \textrm{for all } \  x, y\in B_1$$
and $\| S(0) \| \le K$.
\end{description}

The result is:

\begin{thm}
\label{Apply EK}  With the assumptions above, suppose that $u \in C^2(B_1)$ solves (\ref{me}) and satisfies
\begin{equation}
\label{C2 priori}
S(x) + T\smp{D^2 u (x) , x}  \in \E , \quad \textrm{for all} \ x \in B_1.
\end{equation}
Then $ u  \in C^{2,\alpha} (B_{1/2})$ and
\[
\norm{u}_{C^{2, \alpha} (B_{1/2})} \leq C,
\]
where $\alpha, C$ depend only on $\alpha_0,  K, n, \Lambda, \lambda, \beta,\| f \|_{C^{\alpha_0}}$ and $ \| u \|_{L^{\infty}(B_1)}$.
\end{thm}

The above interior estimate can immediately be applied to each local chart for equations on a compact manifold to obtain a global estimate.

\begin{remark} We also expect a parabolic version of the above theorem to hold (cf. \cite{Wa2}).  This would be useful in providing a unified approach to parabolic H\"older estimates for parabolic flows in complex geometry, including the K\"ahler-Ricci flow, the Chern-Ricci flow, the J-flow and their generalizations (see \cite{Cao, Cha, PSS,  CLN, Gi, TW3, Chen2, We1, FLM} for example).
\end{remark}

The outline of the paper is as follows.  In Section \ref{sectionexamples} we describe how to apply Theorem \ref{Apply EK} to obtain Theorem \ref{theoremCG}.   In Section \ref{sectionc} we briefly describe the relevant Evans-Krylov-Caffarelli theory that we need, and then in Section \ref{sectionproof} we use it to establish Theorem \ref{Apply EK}.  Finally in Section \ref{sectiondon} we describe an application of our results to an equation of Donaldson.

\section{Equations on complex and almost complex manifolds} \label{sectionexamples}

In this section, we give the proof of Theorem \ref{theoremCG}.   First, pick
 a chart in $M$ which is identified with
the unit ball $B_1$ in $\mathbb{C}^n$ with local coordinates $(z^1,\dots,z^n)$. In these coordinates we write
$$\omega=\mn\sum_{i,j} g_{i\ov{j}} dz^i\wedge d\ov{z}^j, $$ where $(g_{i\ov{j}}(x))$ is a positive definite $n\times n$ Hermitian matrix at each point $x\in B_1$. Similarly, we write
$$\chi = \mn \sum_{i,j} h_{i\ov{j}} dz^i\wedge d\ov{z}^j,$$
for $(h_{i\ov{j}}(x))$ an $n\times n$ Hermitian matrix (not necessarily positive definite). Furthermore we have $\ddbar u= \mn\sum_{i,j} u_{i\ov{j}} dz^i\wedge d\ov{z}^j$, where $u_{i\ov{j}}=\frac{\de^2 u}{\de z^i \de\ov{z}^j}$ is
the complex Hessian of $u$.

We also have real coordinates $x^1, \ldots, x^{2n}$ on $B_1$ defined by
$$z^i = x^i + \sqrt{-1} x^{n+i}, \qquad \textrm{for } i=1, \ldots, n.$$
The standard complex structure on $\mathbb{C}^n$ corresponds to an endomorphism $J$ of the real tangent space to $B_1$.  The endomorphism $J$ sends $\partial/\partial x^i$ to $\partial/\partial x^{n+i}$ and $\partial/\partial x^{n+i}$ to $-\partial/\partial x^i$.  As a matrix,
\[
  J= \begin{pmatrix}
0 & -I_n \\
I_n & 0
\end{pmatrix},\]
for $I_n$ the $n\times n$ identity matrix.

In the standard way, we can identify Hermitian $n\times n$ matrices with the subset of $\Sym(2n)$ given by $J$-invariant matrices.
Namely, for $H$ a Hermitian matrix, we write $H=A+\mn B$ with $A,B$ real $n\times n$ matrices, and define
$$\iota(H)=\begin{pmatrix}
A & B \\
-B & A
\end{pmatrix}\in\Sym(2n).$$
Note that if $H_1, H_2$ are two Hermitian matrices then
\begin{equation} \label{eqi}
H_1 \le H_2 \quad \Longleftrightarrow \quad \iota(H_1) \le \iota(H_2).
\end{equation}
Moreover, observe that
$$ \iota\left( 2 u_{i\ov{j}}(x) \right)=\proj(D^2 u(x) ),$$
for
$$\proj(N):= \frac{1}{2} (N + J^TNJ),$$
the projection onto the $J$-invariant part.  The image of $\iota$ is equal to the image of $\proj$.

\bigskip

\noindent
\emph{The complex Monge-Amp\`ere equation (\ref{eqnma})}.
Define $S(x)=\iota(2h_{i\ov{j}}(x))$ and $T(N, x) =  \proj (N)$.  Notice that $T(D^2u(x), x) = \iota( 2u_{i\ov{j}}(x))$.

Next observe that the assumption $\Delta u \le K$ implies that the positive $(1,1)$ form $\chi + \ddbar u$ is bounded from above.  Moreover,
 the equation (\ref{eqnma}) together with the arithmetic-geometric means inequality gives a lower bound for $\chi + \ddbar u$ away from zero.  Hence there is a uniform $C_0$ such that on $B_1$,
\begin{equation} \label{mau}
C_0^{-1} (\delta_{i\ov{j}})  \le 2\left(h_{i\ov{j}} +  u_{i\ov{j}} \right) \le C_0 (\delta_{i\ov{j}}),
\end{equation}
where $(\delta_{i\ov{j}})$ is the $n\times n$ identity matrix, considered as a Hermitian matrix.
Hence  from (\ref{eqi}) we have, for $x\in B_1$,
\begin{equation}
C_0^{-1} I_{2n} \le S(x) + T(D^2u(x), x) \le C_0 I_{2n}.
\end{equation}

We take the convex set $\E$ to be the set of matrices $N\in\Sym(2n)$ with
$$\quad C_0^{-1}I_{2n}\leq N \leq C_0I_{2n},$$
and note that this set is compact.
It is then immediate that
$$S(x) + T(D^2u(x), x) \in \E, \quad \textrm{for all } x \in B_1.$$

We define  $F(N,x)  = \det(N)^{\frac{1}{2n}}$ (independent of $x$) for all $N$ in a small neighborhood of $\E$, extend $F$ arbitrarily to all of $\Sym(2n)\times B_1$, and let $f=2e^{\psi/n}\det(g_{i\ov{j}})^{1/n} \in C^{\alpha_0}$.   Note that for a Hermitian matrix $H$,
$$\det (\iota(H)) = (\det H)^2.$$
Then since $u$ solves (\ref{eqnma}), we have that
\begin{equation}
\begin{split}
F(S(x) + T(D^2u(x),x),x) = {} & \det\left( \iota\left( 2(h_{i\ov{j}} + u_{i\ov{j}})\right)\right)^{1/2n} \\
= {} & \det\left(  2(h_{i\ov{j}} + u_{i\ov{j}})\right)^{1/n} \\
= {} &2e^{\psi/n} \det(g_{i\ov{j}})^{1/n} = f.
\end{split}
\end{equation}

It remains to check the conditions {\bf H1}-{\bf H3}.  For {\bf H1}.(1), note that $F_{ij}(N,x) = \frac{1}{2n} (\det N)^{1/2n} (N^{-1})_{ij}$ for $N\in\E$ and by definition of $\E$ it follows immediately that we have uniform ellipticity with $\lambda$ and $\Lambda$ depending only on $n$ and $C_0$. Concavity of $F$ in $\mathcal{E}$ is well-known (see e.g. \cite{CNS}), giving {\bf H1}.(2).  {\bf H1}.(3) is trivially satisfied.

For {\bf H2}.(3), observe that if $P \in \Sym(2n)$ is nonnegative then for all vectors $v \in \mathbb{R}^{2n}$
$$\langle T(P,x) v, v \rangle = \langle \proj(P) v, v \rangle= \frac{1}{2} \left( \langle Pv, v \rangle + \langle PJv, Jv \rangle \right) \ge 0,$$
and hence $T(P,x)\ge 0$.  Next,
\begin{equation} \label{from}
\frac{1}{2} \| P \| \le \| T(P,x) \| = \frac{1}{2} \sup_{ \| v \| =1} \left( \langle Pv, v \rangle + \langle PJv, Jv \rangle \right) \le \| P \|.
\end{equation}

The other conditions follow easily.  We can then apply Theorem \ref{Apply EK} and obtain $\| u\|_{C^{2,\alpha}(M)}\le C$, as required.

\bigskip
\noindent
\emph{The complex Hessian equations (\ref{eqnch})}.  Recall that if $\lambda = (\lambda_1, \ldots, \lambda_n) \in \mathbb{R}^n$, we define $\sigma_k(\lambda)$ to be the $k^{th}$ elementary symmetric polynomial
$$\sigma_k(\lambda) = \sum_{i_1 < \cdots < i_k} \lambda_{i_1} \cdots \lambda_{i_k}.$$
We can rewrite the equations (\ref{eqnch}) in terms of $\sigma_k$ as follows.
Write $T^{1,0}M$ for the $(1,0)$ part of the complexified tangent space of $M$ (that is, the span of $\frac{\partial}{\partial z^1}, \ldots, \frac{\partial}{\partial z^n}$ over $\mathbb{C}$).  Then we have an endomorphism
$$A: T^{1,0}M \rightarrow T^{1,0}M$$
given by the matrix $A^k_{\ i} = g^{k \ov{j}}(h_{i\ov{j}} + u_{i\ov{j}})$.  Namely, given $X= X^i \frac{\partial}{\partial z^i}$ we define $A(X) = A^k_{\ i} X^i \frac{\partial}{\partial z^k}$.
  The matrix $(A^k_{\ i})$ is Hermitian with respect to the inner product on $T^{1,0}M$ given by $(g_{i\ov{j}})$.  That is, for $X, Y \in T^{1,0}M$, we have
  $$\langle AX, Y \rangle_g = g_{i\ov{j}} A^i_{\ k} X^k \ov{Y^j} = (h_{i\ov{j}} + u_{i\ov{j}}) X^i \ov{Y^j}= g_{i\ov{j}} X^i \ov{A^j_{\ \ell} Y^{\ell}} = \langle X, AY \rangle_g.$$
Now let  $\lambda_1\geq \ldots\geq \lambda_n$ be the eigenvalues of $(A^k_{\ i})$, some of which could be negative.  Then the equation (\ref{eqnch}) can be written as
\begin{equation} \label{eqnch2}
\begin{split}
(\sigma_k(\lambda))^{1/k} = {} &\binom{n}{k}^{1/k} e^{\psi/k}\\
 \sigma_j(\lambda) >{} & 0, \quad \textrm{for } j=1, \ldots, k,
 \end{split}
 \end{equation}
for $\lambda = (\lambda_1, \ldots, \lambda_n)$. Indeed, this is easy to see after picking coordinates for which $(g_{i\ov{j}})$ is the identity, and $(h_{i\ov{j}} + u_{i\ov{j}})$ is diagonal.

To cast \eqref{eqnch2} in the form \eqref{me}, we take a chart in $M$ which is identified with
the unit ball $B_1$ in $\mathbb{C}^n$ with local coordinates $(z^1,\dots,z^n)$.
We let $S(x)=\iota(2h_{i\ov{j}}(x))$ and $T(N, x) =  \proj (N)$, so that $T(D^2u(x), x) = \iota( 2u_{i\ov{j}}(x))$.

Next observe that the assumption $\Delta u \le K$ implies that the eigenvalues $\lambda_j$ of $(A^k_{\ i})$ satisfy $\sigma_1(\lambda)=\lambda_1+\dots+\lambda_n\leq K'$.
We have
$$\frac{\de\sigma_k(\lambda)}{\de \lambda_i}=\sigma_{k-1}(\lambda|i),$$
where $\sigma_{p}(\lambda|i)$ is the $p^{th}$ elementary symmetric function of the $n$-tuple $(\lambda_1,\dots,\lambda_n)$ where we set $\lambda_i=0$.
We will need the following algebraic lemma, which can be easily extracted from the literature (cf. \cite{WaX}).
\begin{lemma}\label{symm}
Let $\lambda=(\lambda_1,\dots,\lambda_n)$ be an $n$-tuple of real numbers which satisfy
\begin{equation} \label{assn1}
\sigma_j(\lambda) >{}  0, \quad \textrm{for } j=1, \ldots, k,
\end{equation}
\begin{equation} \label{assn2}
(\sigma_k(\lambda))^{1/k} \geq {} A^{-1}>0,
\end{equation}
\begin{equation} \label{assn3}
\sigma_1(\lambda)\leq A,
\end{equation}
for some $2\leq k\leq n$, and a constant $A>0$. Then there exists a constant $K_0>0$, which depends only on $A,n,k$, such that
\begin{equation} \label{concl1}
(\sigma_j(\lambda))^{1/j} \geq {} K_0^{-1}>0,\quad \textrm{for } j=1, \ldots, k,
\end{equation}
\begin{equation} \label{concl2}
K_0^{-1}\leq \sigma_{j-1}(\lambda|i)\leq K_0,\quad \textrm{for } i=1,\ldots,n, \textrm{and } j=2,\ldots,k,
\end{equation}
\begin{equation} \label{concl3}
-K_0\leq \lambda_j\leq K_0, \quad \textrm{for } j=1, \ldots, n.
\end{equation}
\end{lemma}
\begin{proof}
Inequality \eqref{concl1} just follows from \eqref{assn1}, \eqref{assn2} and the Maclaurin inequality
$$\sigma_j(\lambda)^{\frac{1}{j}}\geq  C_{n,j,k}\sigma_k(\lambda)^{\frac{1}{k}}\geq C_{n,j,k}A^{-1}.$$
Next we show \eqref{concl2}. Assumption \eqref{assn1} implies that we have for $2 \le j \le k$,
$$\sigma_{j-1}(\lambda|i)\geq \sigma_{j-1}(\lambda|1)>0,$$
see e.g. \cite[Proposition 2.1 (2)]{WaX}.
It is also easy to check (\cite[Proposition 2.1 (1)]{WaX}) that for $2 \le j \le k$,
$$\sum_{i=1}^n\sigma_{j-1}(\lambda|i)=(n-j+1)\sigma_{j-1}(\lambda),$$
hence
\begin{equation}\label{uno}
\sigma_{j-1}(\lambda|i)\leq (n-j+1)\sigma_{j-1}(\lambda)\leq C_{n,j}\sigma_1(\lambda)^{j-1}\leq C_{n,j}A^{j-1},
\end{equation}
using the Maclaurin inequality and \eqref{assn3}.
Using \eqref{assn2} we also have
\begin{equation}\label{due}
\prod_{i=1}^n\sigma_{j-1}(\lambda|i)\geq C_{n,j}\sigma_{j}(\lambda)^{\frac{n(j-1)}{j}} \geq C_{n,j,k} \sigma_k(\lambda)^{\frac{n(j-1)}{k}} \geq C_{n,j,k}A^{-n(j-1)},
\end{equation}
where the first inequality is \cite[Proposition 2.1 (4)]{WaX}.
Combining \eqref{uno} and \eqref{due}, we get \eqref{concl2}.

Finally, we prove \eqref{concl3}. Since $k\geq 2$, inequalities \eqref{concl2} hold with $j=2$,
$$K_0^{-1}\leq \sigma_{1}(\lambda|i)\leq K_0.$$
But we clearly have
$$\lambda_i=\sigma_1(\lambda)-\sigma_1(\lambda|i),$$
and the bound \eqref{concl3} follows.
\end{proof}
Thanks to \eqref{eqnch2}, the eigenvalues $\lambda_j$ of $(A^k_{\ i})$ satisfy the hypotheses of Lemma \ref{symm}, for some uniform constant $A$ (on the whole of $B_1$). Let $K_0$ be the constant that
we obtain from Lemma \ref{symm}.

Given a matrix $N\in\Sym(2n)$, then $\iota^{-1}(p(N))$ is an $n\times n$ Hermitian matrix. Denote by
$A(N)^k_{\ i}=g^{k \ov{j}}(0) \left(\iota^{-1}(p(N))\right)_{i\ov{j}}$, which is Hermitian with respect to the inner
product $g_{i\ov{j}}(0)$, with eigenvalues $\lambda(0)=(\lambda_1(0),\dots,\lambda_n(0))$. Let $\E$ be the convex set of matrices $N\in\Sym(2n)$ with
\begin{equation}\label{tre}
\begin{split}
-2K_0 \leq \lambda_i(0) \leq {} &2 K_0, \ \textrm{for } 1 \le i \le n \\
 (\sigma_j(\lambda(0))^{1/j} \ge {} & (2K_0)^{-1}, \ \textrm{for } 1\le j \le k.
 \end{split}
\end{equation}
The set $\E$ is compact, and the concavity of $\sigma_j^{1/j}$   (see \cite{CNS}, for example) implies that $\E$ is convex.

Now write $\lambda(x)= (\lambda_1(x),\dots,\lambda_n(x))$ for the eigenvalues of $A(N,x)^k_{\ i}=g^{k \ov{j}}(x) \left(\iota^{-1}(p(N))\right)_{i\ov{j}}$.
For $N$ in a sufficiently small neighborhood of $\E$ we define
$$F(N,x)=(\sigma_k(\lambda(x)))^{1/k},$$
and extend $F$ arbitrarily to $\Sym(2n)\times B_1$.
Since $g_{i \ov{j}}(x)$ is a continuously-varying Hermitian matrix, for all $N \in \E$ we have that $\lambda(x)$ satisfies
\begin{equation}\label{tre2}
\begin{split}
-4K_0 \leq \lambda_i(x) \leq {} & 4K_0, \ \textrm{for } 1 \le i \le n \\
 (\sigma_j(\lambda(x))^{1/j} \ge {} & (4K_0)^{-1}, \ \textrm{for } 1\le j \le k.
 \end{split}
\end{equation}
 as long as $x$ is sufficiently close to $0$.  If necessary, we replace the ball $B_1$ with a smaller ball $B_r$ (for some uniform $0<r<1$) to ensure that this holds.  Of course, by a simple covering argument this does not affect the conclusion.

From what we have shown above, if $u$ solves \eqref{eqnch} on $M$ then the eigenvalue $\lambda(x)$ corresponding to the matrix $N=S(x)+T(D^2 u(x),x)$ satisfies (\ref{tre2}) with $4K_0$ replaced by $K_0$. Again using the continuity of
$g_{i \ov{j}}(x)$, it follows from the definition of $\mathcal{E}$ that, after possibly shrinking $r$, we have
$$S(x)+T(D^2 u(x),x)\in\E,$$
for all $x\in B_r$.

By the concavity of $\sigma_k^{1/k}$, we see that for each fixed $x$, $F(N,x)$ is concave in $\E$.  Moreover, thanks to Lemma \ref{symm},  (\ref{tre2}) implies bounds of the form
$$(K'_0)^{-1}\leq \sigma_{k-1}(\lambda|i)\leq K'_0, \quad 1 \le i \le n,$$
so we see that for $N \in \E$, $F(N,x)$ is uniformly elliptic.
Since $g_{i \ov{j}}(x)$ is a smoothly-varying Hermitian matrix, it is easy to see that, since $\E$ is compact,
$$|F(N,x)-F(N,y)|\leq K|x-y|^\beta,$$
for all $N\in\E$ and all $x,y\in B_1$, giving {\bf H1}.(3).

We can then apply Theorem \ref{Apply EK} and obtain $\| u\|_{C^{2,\alpha}(M)}\le C$, as required.

\bigskip
\noindent
\emph{The complex $\sigma_n/\sigma_k$ equations (\ref{eqnd})}.  The argument for the equations (\ref{eqnd}) is similar, and slightly simpler, than the argument given above for the complex Hessian equations.  Indeed, using the same notation as there, we may write (\ref{eqnd}) in terms of the eigenvalues $\lambda=(\lambda_1, \ldots, \lambda_n)$ as
\begin{equation} \label{eqnd2}
\left( \frac{\sigma_n}{\sigma_k} \right)^{1/(n-k)} = e^{-\psi/(n-k)} \binom{n}{k}^{-1/(n-k)}, \quad \lambda_1 \ge \cdots \ge  \lambda_n >0.
\end{equation}
From the assumption $\Delta u \le K$ we have $\sigma_1(\lambda)\le K'$ and hence all the eigenvalues $\lambda_i$ are bounded from above by $K'$.  Moreover, the equation (\ref{eqnd2}) implies lower bounds for $\lambda_i$ away from zero.  Indeed, from (\ref{eqnd2}) we have an upper bound for
$$ \frac{ \sum_{i_1 < \cdots < i_k} \lambda_{i_1} \cdots \lambda_{i_k}}{\lambda_1 \lambda_2 \cdots \lambda_n},$$
and since each term in the sum is positive we obtain in particular,
$$\frac{1}{\lambda_{k+1} \lambda_{k+2} \cdots \lambda_n} \le C.$$
Hence for each $i=1, \ldots, n$,
$$\lambda_i \ge \lambda_n \ge \frac{1}{C\lambda_{k+1} \cdots \lambda_{n-1}} \ge c >0,$$
for a uniform $c>0$, since $\lambda_{k+1}, \cdots, \lambda_{n-1}$ are all uniformly bounded from above.

Now define $\E$ to be the compact convex set of matrices $N \in \Sym(2n)$ whose eigenvalues $\lambda_i(0)$ (with respect to $g(0)$) satisfy
$$K_0^{-1} \le \lambda_i(0) \le K_0,$$
for a sufficiently large $K_0$.
The operator
$$F(N, x) = \left( \frac{\sigma_n(\lambda(x))}{\sigma_k(\lambda(x))}\right)^{1/(n-k)}$$ is concave (see e.g. \cite{Sp}) and uniformly elliptic on $\E$ (after possibly shrinking the neighborhood), and hence we can apply Theorem
\ref{Apply EK} in the same way as for the complex Hessian equations above.

\begin{remark} In fact, by a similar argument, Theorem \ref{theoremCG} applies to the more general complex $\sigma_\ell/\sigma_k$ equations
\begin{equation} \label{eqnsigma}
\begin{split}
(\chi + \ddbar u)^k \wedge \omega^{n-k} = {} & e^{\psi} (\chi + \ddbar u)^\ell\wedge\omega^{n-\ell} \\
(\chi + \ddbar u)^j\wedge \omega^{n-j} > {} & 0, \quad \textrm{for } j=1, \ldots, \ell,
\end{split}
\end{equation}
where $1<k<\ell\leq n$.
\end{remark}

\bigskip
\noindent
\emph{The Monge-Amp\`ere equation for $(n-1)$-psh functions (\ref{eqnnm1})}.   First note that the assumption $\Delta u \le K$ implies that the positive definite $(1,1)$ form
$$\chi + \frac{1}{n-1} \left(  (\Delta u) \omega - \ddbar u \right) >0$$
has bounded trace with respect to $g$.  The equation (\ref{eqnnm1}) then implies that this positive definite $(1,1)$ form is uniformly bounded from above, and below away from zero.

We define $S(x) = \iota (2 h_{i\ov{j}}(x))$.  We let $g(x) = \iota (2 g_{i\ov{j}}(x))$
be the Riemannian metric associated to $(g_{i\ov{j}})$ and we define
$$T(N,x)=\frac{1}{n-1}\left(\frac{1}{2}\trace(g(x)^{-1} \proj(N)) g(x)-\proj(N)\right).$$
As with the complex Monge-Amp\`ere equation above, define $F$ by
$F(N,x)  = \det(N)^{\frac{1}{2n}}$ (independent of $x$) for $N$ in a neighborhood of $\E$, extend $F$ arbitrarily to $\Sym(2n)\times B_1$, and observe that equation (\ref{eqnnm1}) is equivalent to
$$F(S(x) + T(D^2 u(x), x),x) = 2 e^{\psi/n} \det (g_{i\ov{j}})^{1/n} = : f.$$

Take the convex set $\E$ to be the set of matrices $N\in\Sym(2n)$ with
$$K_0^{-1}I_{2n}\leq N\leq K_0I_{2n},$$
for a sufficiently large constant $K_0$.

We check that the hypotheses of Theorem \ref{Apply EK} are satisfied. Clearly  {\bf H1},  {\bf H3} and  {\bf H2}.(2) hold.
For {\bf H2}.(1) we have
\[
\begin{split}
\lefteqn{
\norm{ T (N, x) -  T(N, y) }  } \\& \leq \frac{1}{(n-1)}\left( \frac{1}{2}\trace  ((g^{-1}(x)-g^{-1}(y))\proj(N))g(x) - \frac{1}{2}\trace(g^{-1}(y)\proj(N))(g(y)-g(x))\right) \\
& \leq K \abs{x -y}^\beta \norm{N} .
\end{split}
\]
Lastly, we check {\bf H2}.(3). If $P \geq 0$, then $\frac{1}{2}\trace(g(x)^{-1} \proj(P)) g(x)-\proj(P) \geq 0,$
because at any $x$ we can choose a basis such that $g(x)=I$ while $\proj(P)$ is diagonal with eigenvalues $\lambda_1,\lambda_1,\dots,\lambda_n,\lambda_n\geq 0$, and then
the eigenvalues of $\frac{1}{2}\trace(g(x)^{-1} \proj(P)) g(x)-\proj(P)$ are $\sum_{i\neq j}\lambda_i\geq 0$. From this we also see that
\[
\frac{1}{n-1} \norm{\proj(P)}\leq \frac{1}{n-1}\left\| \frac{1}{2}\trace(g(x)^{-1} \proj(P)) g(x)-\proj(P)\right\| \leq  \norm{\proj(P)}.
\]
But $P\geq 0$ implies from (\ref{from}) that $\|\proj(P)\|\leq \|P\|\leq 2\|\proj(P)\|$, which gives {\bf H2}.(3).

We can then apply Theorem \ref{Apply EK} and obtain $\| u \|_{C^{2, \alpha}(M)} \le C$.

\bigskip
\noindent
\emph{The $(n-1)$-plurisubharmonic version of the complex Hessian and $\sigma_n/\sigma_k$ equations \eqref{eqnnm2}}.
This case follows easily by combining the arguments for the previous three examples,
choosing $S(x), T(N,x)$ as in the discussion of \eqref{eqnnm1}, and $F(N,x), \mathcal{E}$ as in the discussions of \eqref{eqnch} and \eqref{eqnd} respectively.

\bigskip
\noindent
\emph{The almost complex case (\ref{eqnac}).}
Here we assume more generally that $(M,J)$ is an almost-complex manifold. Then one can define
$(dJdu)^{(1,1)}$, which takes the place of $\ddbar u$. The convention we use for the action of $J$ on $1$-forms is
$J\alpha(X):=-\alpha(JX)$, for any $1$-form $\alpha$ and vector $X$. With this convention, we have $(dJdu)^{(1,1)}=2\ddbar u$
when $J$ is integrable. For notational convenience, in the following we will work with $(dJdu)^{(1,1)}$ instead of $\frac{1}{2}(dJdu)^{(1,1)}$, which will not affect any of the results.
We first deal with the case of the Monge-Amp\`ere equation (\ref{eqnma}) in this setting.

In this case, we define the Laplacian of $u$ to be
$$\Delta u = \frac{n  \omega^{n-1} \wedge (dJdu)}{\omega^n}=\frac{n  \omega^{n-1} \wedge (dJdu)^{(1,1)}}{\omega^n},$$
which differs from the usual Riemannian Laplacian by a first order term.

We fix a point $p\in M$ and choose a real coordinate system
 $x^1, \ldots, x^{2n}$ centered at $p$.  Writing $J= (J^k_{\ \ell})$ we see that
\begin{equation}\label{dJ}
dJdu = \left(- J^k_{\ \ell} \partial_k \partial_{i} u  - (\partial_{i} J^k_{\ \ell} ) \partial_k u \right) dx^{i} \wedge dx^{\ell}.
\end{equation}
Its $(1,1)$ part is given by
\[
\begin{split}
(dJdu)^{(1,1)} = {} & \frac{1}{2} \left(
- J^k_{\ \ell} \partial_k \partial_{i} u - J^a_{\ i} J^b_{\ \ell} J^k_{\ b} \partial_k \partial_a u  \right. \\ & \left. - (\partial_{i} J^k_{\ \ell} ) \partial_k u - J^a_{\ i} J^b_{\ \ell} (\partial_a J^k_{\ b}) \partial_k u \right) dx^i \wedge dx^{\ell}.
\end{split}
\]
For later use, we note that
\begin{equation}\label{20}
(dJdu)^{(2,0)+(0,2)} = \frac{1}{2} \left( - (\partial_{i} J^k_{\ \ell} ) \partial_k u + J^a_{\ i} J^b_{\ \ell} (\partial_a J^k_{\ b}) \partial_k u \right) dx^i \wedge dx^{\ell}.
\end{equation}
Associated to $(dJdu)^{(1,1)}$ is a symmetric bilinear form $H(u)$ defined by $H(u) (X, Y) = (dJdu)^{(1,1)}( X, J Y)$.  Compute
$$(H(u))_{ij} := \frac{1}{2} \left( \partial_i \partial_j u + J^k_{\ i} J^{\ell}_{\ j} \partial_k \partial_{\ell} u \right) + E_{ij},$$
where
$$E_{ij} = \frac{1}{4} \left( - J^{\ell}_{\ j} (\partial_i J^k_{\ \ell} ) - J^{\ell}_{\ i} (\partial_j J^k_{\ \ell})  + J^{\ell}_{\ j} (\partial_{\ell} J^k_{\ i} )  + J^{\ell}_{\ i } (\partial_{\ell} J^k_{\ j}) \right) \partial_k u.$$
Or, in other words,
$$H(u)(x)=\proj(D^2 u(x),x)+E,$$
where the ``error'' $E$ depends linearly on $Du$ (cf. \cite[Proposition 4.2]{HL}). Here we denote by
$$\proj(N,x)=\frac{1}{2} (N + J^T(x)NJ(x)).$$
We set $T(N,x) = \proj(N,x)$.  Observe that $T$ satisfies the properties {\bf H2}.(1) and {\bf H2}.(2).  Moreover,
we may pick coordinates $x^1, \ldots, x^{2n}$ so that at the origin,
$$  J(0)= \begin{pmatrix}
0 & -I_n \\
I_n & 0
\end{pmatrix},$$
the standard complex structure.  By shrinking the neighborhood if necessary, $J(x)$ is only a small perturbation of $J(0)$ and so $T(N,x)$ also satisfies the property {\bf H2}.(3) by the same argument as above.

We define $S$ and $F$ in the obvious way,  and thus we have essentially the same setup as in the complex Monge-Amp\`ere equation (\ref{eqnma}) with the exception of the extra term $E$.  However, this term is linear in $Du$.  The uniform bounds that we assume on $u$ and $\Delta u$ imply a $C^{1,\beta}$ bound for $u$ for any $0<\beta<1$.  Hence the $E$ term is bounded in $C^{\beta}$ and can then be absorbed in the term $\chi$.

The arguments for the other equations (\ref{eqnch}), (\ref{eqnd}), (\ref{eqnnm1}) and (\ref{eqnnm2}) in the almost complex case follow similarly.

\section{Evans-Krylov Theory on Euclidean Space} \label{sectionc}

In this section, we recall Evans-Krylov theory on Euclidean space and its perturbation version due to Caffarelli. Again in this section $B_1$ will denote the unit ball in $\mathbb{R}^{2n}$.

First, recall the following version of Evans-Krylov theorem (see \cite[Theorem 6.6]{CC} and also \cite{CS, Ev, Kr}).
\begin{thm}
\label{EK local}
Assume that $F :\Sym(2n) \rightarrow \R$ is a concave function and uniformly elliptic, i.e.,
\[
\lambda \norm{P} \leq F(N +P ) - F(N )  \leq \Lambda \norm{P} ,\quad \forall N , P \in \Sym(2n) , P \geq 0.
\]
If $F(0) =0 $ and a continuous function $u: B_1 \rightarrow \R$ satisfies
\[
F (D^2 u )  =0 ,\quad \text{ in the viscosity sense.}
\]
Then $u \in C^{2,\beta} (B_{1/2})$ and
\[
\norm{u}_{C^{2,\beta} (B_{1/2})} \leq C \norm{u}_{L^{\infty}(B_1)}
\]
where $\beta \in (0,1)$ and $C$ only depend on $n, \lambda, \Lambda$.
\end{thm}

For our application, we need a more general version of the above theorem. We first introduce the following definition.

\begin{definition}
Let $\mathscr{F}_{2n} (\lambda, \Lambda, K, \gamma)$ be a family of functions $\Phi : \Sym(2n)  \times B_1 \rightarrow \R$ depending on positive constants $\lambda, \Lambda, K$ and $\gamma \in (0,1)$. An element $\Phi  \in \mathscr{F}_{2n} (\lambda, \Lambda, K, \gamma)$ satisfies the following conditions:

\begin{itemize}
\item \emph{Fiberwise concavity.}  For each  fixed $x \in B_1$,
\[
\Phi \smp{ \frac{A+ B}{2} , x }\geq  \frac{1}{2} \Phi(A, x ) + \frac{1}{2} \Phi(B,x), \quad \textrm{for all } A, B \in \Sym(2n).
\]

\item \emph{Uniform Ellipticity.} For all $x\in B_1$ and all $N,P \in \Sym(2n)$ with $P \ge 0$ we have
\[
 \lambda \norm{P} \leq \Phi (N +P ,  x) - \Phi(N,  x) \leq \Lambda \norm{P}.
\]

\item \emph{H\"older bound in $x$}. For all $x,y \in B_1$ and all $N \in \Sym(2n)$,
\[
 \frac{|\Phi(N, x)  - \Phi(N,  y)|}{\|N\| +1} \leq K \abs{x-y}^{\gamma}.\]

\end{itemize}
\end{definition}

\begin{remark}
If $\Phi \in \mathscr{F}_{2n} (\lambda, \Lambda, K, \gamma)$ and the equation $\Phi(D^2 u(x) , x) = 0$ is a linear equation, then the linear equation takes the form
\[
\sum_{i,j} a_{ij}(x) u_{ij}(x)  =0
\]
with $\norm{a_{ij} (x) }_{C^{\gamma } } \leq K$ and the eigenvalues of $(a_{ij}(x))$ lie between $\lambda$ and $\Lambda/2n$.
\end{remark}

Now, we state the Evans-Krylov theorem for $\Phi \in \mathscr{F}_{2n} (\lambda, \Lambda, K, \gamma)$.

\begin{thm}
\label{EKC local}
Assume that $\Phi   \in \mathscr{F}_{2n} (\lambda, \Lambda, K, \gamma)$ and $f \in C^{\alpha_0} (B_1)$.
If $u \in C^2(B_1)$ satisfies
\[
\Phi(D^2 u(x) ,  x) = f(x)  \quad \text{ in } B_1
\]
then $ u \in C^{2,\alpha } (B_{1/2}) $ and
\[
\norm{u }_{C^{2,\alpha} (B_{1/2})} \leq C
\]
where $C, \alpha$ depend only on $\alpha_0, K, \gamma, n, \lambda, \Lambda, \| f\|_{C^{\alpha_0}(B_1)}, \| u\|_{L^{\infty}(B_1)}$ and $\Phi(0,0)$.
\end{thm}

Indeed this theorem is a direct consequence of a theorem of Caffarelli \cite[Theorem 3]{Ca}. This result is also used in the second-named author's paper \cite[Corollary 2.3]{Wa}, in the case when $\Phi$ does not depend on $x$. Furthermore, Theorem \ref{EKC local} remains true if $u\in C^0(B_1)$ is just a viscosity solution, and the proof is the same.

For the reader's convenience,
we recall the statement of Caffarelli's theorem, following the exposition in \cite[Theorem 8.1]{CC}.
Suppose that we have a viscosity solution of the equation
\begin{equation}\label{eqn}
G(D^2u(x),x)=f(x), \ \ \textrm{for } x\in B_1,
\end{equation}
where $u\in C^0(B_1),$  $G:\Sym(2n)\times B_1\to \mathbb{R}$ is continuous in $x$, $f$ is continuous, $G(0,0)=f(0)=0$, and $G$ satisfies the uniform ellipticity condition
\begin{equation}\label{ell}
\lambda \norm{P} \leq G (N +P ,  x) - G(N,  x) \leq \Lambda \norm{P}, \ \ \textrm{for } x\in B_1,
\end{equation}
for some positive constants $\lambda, \Lambda$, and for all $N,P\in\Sym(2n)$ with  $P\geq 0$.
Let $$\ti{\beta}(x)=\sup_{N\in\Sym(2n)}\frac{|G(N, x)  - G(N,  0)|}{\|N\| +1}.$$
Furthermore assume that the following hypotheses hold:\\

\noindent
{\bf Hypothesis 1. }There are constants $0<\ov{\alpha}<1$ and $C_e>0$ such that for any symmetric matrix $N$ with $G(N,0)=0$ and any $w_0\in C^0(\de B_1)$, there exists a function
$w\in C^2(B_1)\cap C^0(\ov{B_1})\cap C^{2,\ov{\alpha}}(B_{1/2})$ which solves
\begin{equation}\label{dir}
G(D^2w(x)+N,0)=0
\end{equation}
in $B_1$ and $w=w_0$ on $\de B_1$, and
\begin{equation}\label{todo}
\|w\|_{C^{2,\ov{\alpha}}(B_{1/2})}\leq C_e\|w\|_{L^\infty(B_1)}.
\end{equation}

\noindent
{\bf Hypothesis 2. }There exist $0<\alpha<\ov{\alpha}$,  $0<r_0\leq 1, C_1>0$ and $C_2>0$, such that
$$\left(\frac{1}{r^{2n}}\int_{B_r} \ti{\beta}^{2n}\right)^{1/2n}\leq C_1 r^\alpha r_0^{-\alpha},$$
$$\left(\frac{1}{r^{2n}}\int_{B_r} |f|^{2n}\right)^{1/2n}\leq C_2 r^\alpha r_0^{-\alpha},$$
for all $0<r\leq r_0$.\\

Then Theorem 8.1 of \cite{CC} asserts:

\begin{thm}\label{caffa}
Let $u$ solve (\ref{eqn}) with the assumptions described above, including   Hypothesis 1 and Hypothesis 2.
Then $u$ is $C^{2,\alpha}$ at the origin.  More precisely, there exists a constant $C>1$ which depends only on $n$, $\lambda$, $\Lambda$, $C_e$, $\ov{\alpha}$, $\alpha$, $C_1$, $C_2$ and
a polynomial function $P$ of degree $2$ such that
$$\|u-P\|_{L^\infty(B_r)}\leq C_3 r^{2+\alpha} r_0^{-(2+\alpha)},\quad \textrm{ for all }r\leq C^{-1}r_0,$$
$$r_0|DP(0)|+r_0^2\|D^2P\|\leq C_3,$$
$$C_3\le C(\|u\|_{L^\infty(B_{r_0})}+r_0^2(C_2+1)).$$
\end{thm}

We now use Theorems \ref{EK local} and \ref{caffa} to prove Theorem \ref{EKC local}.

\begin{proof}[Proof of Theorem \ref{EKC local}] We follow the exposition in \cite[Corollary 2.3]{Wa}.
Define $$G(N,x)=\Phi(N+t_0 I_{2n},x)-f(0),$$ where $t_0\in\mathbb{R}$ is chosen so that
$\Phi(t_0 I_{2n},0)=f(0)$. Such a $t_0$ exists because of the ellipticity of $\Phi(N,0),$ which also implies that
\begin{equation}\label{elli}
|t_0|\leq\lambda^{-1}|\Phi(0,0)-f(0)|.
\end{equation}
We also define $g(x)=f(x)-f(0)$ and $v(x)=u(x)-\frac{t_0}{2}|x|^2$. Then we have that
$$G(D^2v(x),x)=g(x),$$
and $G(0,0)=g(0)=0$. Thanks to \eqref{elli}, it is enough to bound the $C^{2,\alpha}$ norm of $v$.

It is clear that $G$ satisfies the ellipticity condition \eqref{ell} and is continuous in $x$. To verify Hypothesis 1, it is enough to observe that since $G$ is uniformly elliptic, the Dirichlet problem for a viscosity solution of the equation \eqref{dir} can
be solved using Perron's method, and the Evans-Krylov Theorem \ref{EK local} shows that $u$ is $C^{2,\ov{\alpha}}$ with the estimate \eqref{todo}, with $C_e$ and $\ov{\alpha}$ depending only on $n,\lambda, \Lambda$.

Furthermore, we have
$$\ti{\beta}(x)=\sup_{N\in\Sym(2n)}\frac{|\Phi(N+t_0 I_n, x)  - \Phi(N+t_0 I_n,  0)|}{\|N\| +1}\leq C|x|^\gamma,$$
where $C$ depends on $K$, $n$, $\lambda$, $\Phi(0,0)$ and $f(0)$. It follows that Hypothesis 2 is satisfied for any $r_0\leq 1$ and any $\alpha<\min(\gamma,\alpha_0,\ov{\alpha})$,
with $C_1,C_2$ depending only on  $K$, $n$, $\lambda$, $\Phi(0,0)$, $f(0)$ and $\|f\|_{C^{\alpha_0}(B_1)}$.

We can thus apply Theorem \ref{caffa} and conclude that $v$ (and hence $u$) is $C^{2,\alpha}$ at $0$ with the bounds given there. By translation of the coordinates, we see that $u$ is
$C^{2,\alpha}$ at every $x\in B_{1/2}$ with these bounds. But a standard covering argument (see e.g. \cite[Remark 3, p.74]{CC}) implies that $u\in C^{2,\alpha}(B_{1/2})$ with
$\norm{u }_{C^{2,\alpha} (B_{1/2})} \leq C$ for a constant $C$ depending only on the stated quantities.
\end{proof}

\section{Proof of Theorem \ref{Apply EK}} \label{sectionproof}
In this section, we present the proof of Theorem \ref{Apply EK}.

\begin{proof}[Proof of Theorem \ref{Apply EK}]
Under the assumptions of Theorem \ref{Apply EK}, we shall construct a map $\Phi : \Sym(2n)  \times B_1 \rightarrow \R$ in $\mathscr{F}_{2n} (K^{-1}\lambda, 2Kn\Lambda, 5\Lambda Kn, \beta)$ such that
\[
F(S(x) + T(D^2u(x), x),x) = \Phi(D^2u(x), x)  \quad \textrm{for all } x \in B_1,
\]
and then apply Theorem \ref{EKC local}.

First of all, recall that since $\langle A, B \rangle = \textrm{tr}(AB)$ defines an inner product on the space of symmetric matrices, every linear map $\Sym(2n) \rightarrow \mathbb{R}$ can be written as $X \mapsto \textrm{tr}(AX)$ for some symmetric matrix $A \in \Sym(2n)$.  In this way, we identify the derivative $DG$ of any map $G:\Sym(2n) \rightarrow \mathbb{R}$ at a given point with an element of $\Sym(2n)$.

Let also $\mathcal{H}$ be subset of $\mathrm{Sym}(2n)$ given by symmetric matrices with all eigenvalues in the interval $[\lambda,\Lambda]$. By assumption {\bf H1}.(1), for any $x\in B_1$ and $A\in \E$, the derivative $DF = (F_{ij})$ of $F(\cdot, x)$ at $A$ lies in $\mathcal{H}.$

Following \cite{Wa}, we define $\bar{F}: \Sym(2n)\times B_1 \rightarrow \mathbb{R}$ as follows
\[\begin{split}
\bar{F} (N,x) := \inf\{ L(N)  \ | \ L: \Sym(2n) \rightarrow \mathbb{R} &\text{ affine linear},\, DL\in\mathcal{H},\\
& L(A) \geq F (A,x) , \forall A \in \E  \}.
\end{split}\]
Note that $\bar{F}(N,x)>-\infty$ for all $N\in\Sym(2n), x\in B_1$.
We have the following lemma (cf. \cite[Section 3]{Wa}).

\begin{lemma}\label{use}
$\bar{F}$ satisfies the following properties.
\begin{enumerate}
\item[(i)] $\bar{F}(\cdot, x)$ is concave on $\Sym(2n)$ for all $x\in B_1$, and $\bar{F}= F$ on $\E\times B_1$.
\item[(ii)] $\bar{F}(\cdot,x)$ is Lipschitz on $\Sym(2n)$ with Lipschitz constant $2n \Lambda$, for all $x\in B_1$.
\item[(iii)] For all $N, P \in \Sym(2n)$ with $P \ge 0$ and for all $x\in B_1$, we have
$$ \lambda \| P \| \le \bar{F}(N+P,x) - \bar{F}(N,x) \le 2n \Lambda \| P \|.$$
\item[(iv)] For all $N\in\Sym(2n)$ and all $x,y\in B_1$ we have
$$|\bar{F}(N,x)-\bar{F}(N,y)|\leq K|x-y|^\beta.$$
\end{enumerate}

\end{lemma}
\begin{proof}
Part (i) is immediate.  Indeed, $\bar{F}(\cdot,x)$ is concave on $\Sym(2n)$ since it is an infimum of affine linear functions.  From {\bf H1}, $F(\cdot,x)$ is concave on $\E$
and its derivative $DF$ at a point $A\in \E$ lies in $\mathcal{H}$, it follows that $\bar{F}=F$ on $\E\times B_1$.


Every affine linear map $L : \Sym(2n) \rightarrow \mathbb{R}$ can be written as $$L(X) = \textrm{tr}(AX) +c$$ for  $A=DL \in \Sym(2n)$ and some constant $c \in \mathbb{R}$.  Now for any $N, X \in \Sym(2n)$ we will show that
$$|\bar{F}(N+X,x) - \bar{F}(N,x)| \le 2n \Lambda \| X \|,$$
which will establish (ii).  From the definition of $\bar{F}$ it follows easily that there exist affine linear maps $L_1, L_2 : \Sym(2n) \rightarrow \mathbb{R}$ with $A_1=DL_1,A_2=DL_2\in\mathcal{H}$, $L_1, L_2 \ge F(\cdot,x)$ on $\E$ and
$$\bar{F}(N+X,x) = L_1(N+X), \quad \bar{F}(N,x) = L_2(N).$$
Hence
\[
\begin{split}
\bar{F}(N+X,x) -\bar{F}(N,x) = {} & L_1(N+X) - L_2(N) \\
\le {} & L_2(N+X) - L_2(N) \\
= {} & \textrm{tr}(A_2 X)
\le  2n \Lambda \| X \|,
\end{split}
\]
where the first inequality uses the fact that $L_1(N+X)$ is an infimum of $L(N+X)$ over all affine linear functions $L$ with $L|_{\E} \ge F(\cdot,x)|_{\E}$.  Similarly,
\[
\begin{split}
\bar{F}(N+X,x) - \bar{F}(N,x) = {} & L_1(N+X) - L_2(N) \\
\ge {} & L_1(N+X) - L_1(N) \\
= {} & \textrm{tr}(A_1X)  \ge - 2n\Lambda \| X \|.
\end{split}
\]

For (iii), all that  remains is to prove the lower bound of $\bar{F}(N+P,x) - \bar{F}(N,x)$ for $P \ge 0$.  Write $\mu_i\ge 0$ for the eigenvalues of $P$.
Then if $\bar{F}(N+P,x) = L_1(N+P)$ and $\bar{F}(N,x) = L_2(N)$, we have, using the notation and argument as above,
\[
\begin{split}
\bar{F}(N+P,x) - \bar{F}(N,x)
\ge {} & \trace(A_1P)
\ge  \lambda (\mu_1 + \cdots + \mu_{2n}) \ge \lambda \| P \|.
\end{split}
\]
For (iv), given $N\in\Sym(2n)$ there exist affine linear maps $L_1, L_2 : \Sym(2n) \rightarrow \mathbb{R}$ with $DL_1,DL_2\in\mathcal{H}$, $L_1 \ge F(\cdot,x)$ and $L_2\geq F(\cdot,y)$ on $\E$ and
$$\bar{F}(N,x) = L_1(N), \quad \bar{F}(N,y) = L_2(N).$$
For any given $A\in \E$ we have
$$L_2(A)\geq F(A,y)\geq F(A,x)-K|x-y|^\beta,$$
using hypothesis {\bf H1}.(3). Therefore $\ti{L}_2:=L_2+K|x-y|^\beta$ is another affine linear map with $D\ti{L}_2\in\mathcal{H}$ and
$\ti{L}_2\geq F(\cdot, x)$ on $\E$. By definition of $\bar{F}$, we have
$$L_1(N)\leq \ti{L}_2(N)=L_2(N)+K|x-y|^\beta,$$
i.e. $L_1(N)-L_2(N)\leq K|x-y|^\beta$. Similarly we obtain $L_2(N)-L_1(N)\leq K|x-y|^\beta$, and so
$$|\bar{F}(N,x)-\bar{F}(N,y)|=|L_1(N)-L_2(N)|\leq K|x-y|^\beta.$$
This completes the proof.
\end{proof}

Now,  define $\Phi : \Sym (2n)  \times B_1 \rightarrow \mathbb{R}$  by
\[
\Phi (N, x) := \bar{F} \smp{ S(x) + T(N,x),x},
\]
noting that $\Phi(0,0)=\bar{F}(S(0),0)$ is bounded by {\bf H1} and {\bf H3}.

\begin{lemma}
\label{Claim2 aEK}
$\Phi$ lies in $\mathscr{F}_{2n} (K^{-1}\lambda, 2Kn\Lambda, 5\Lambda Kn, \beta)$.
\end{lemma}

Given the lemma, we can complete the proof of Theorem \ref{Apply EK}.  Indeed, by the assumption \eqref{C2 priori}, we know
\[
S(x)+  T(D^2 u (x) , x)    \in \E , \quad \forall x \in B_1.
\]
It follows then from Lemma \ref{use} (i) that for $x\in B_1$
$$\Phi(D^2u(x), x) =F(S(x)+T(D^2u(x),x),x)= f(x).$$
and applying Theorem \ref{EKC local} completes the proof.
\end{proof}

\begin{proof}[Proof of Lemma \ref{Claim2 aEK}]

First, we check fiberwise concavity of $\Phi$. Fix $x \in B_1$.  Then for any $A,B \in \Sym(2n)$ we have
\[
\begin{split}
\Phi\smp{ \frac{A +B }{2} , x} & = \bar{F}  \smp{ S (x) + T \left( \frac{A+B}{2}, x \right) ,x  } \\
&  = \bar{F} \smp{ \frac{S(x)  + T(A, x)   }{2} + \frac{S( x) + T(B,x)  }{2}  ,x } \\
& \geq \frac{1}{2}\bar{F} \smp{ S(x)  + T(A,x)   ,x }  + \frac{1}{2}\bar{F} \smp{ S(x) + T(B, x)   ,x} \\
& = \frac{1}{2} \Phi (A, x) +  \frac{1}{2} \Phi (B, x).
\end{split}
\]
Here, we have used linearity of $T$ and concavity of $\bar{F}$.

For uniform ellipticity of $\Phi$, we compute for $x \in B_1$ and any $N,P \in \Sym(2n)$ with $P\ge 0$,
\[
\begin{split}
\lefteqn{\Phi(N+P,x) - \Phi(N,x)  }\\  = {} & \bar{F}(S(x)+T(N+P,x),x) - \bar{F}(S(x)+T(N,x),x) \\
= {} & \bar{F}(S(x)+ T(N,x) + T(P,x),x)  -  \bar{F}(S(x)+T(N,x),x).
\end{split}
\]
From the ellipticity of $\bar{F}$ we have
$$\lambda \| T(P,x) \| \le \Phi(N+P,x) - \Phi(N,x) \le 2n \Lambda \| T(P,x) \|$$
and hence from {\bf H2}.(3),
$$K^{-1} \lambda \| P \| \le   \Phi(N+P,x) - \Phi(N,x)  \le 2Kn \Lambda \| P\|,$$
giving the required ellipticity of $\Phi$.

Finally, we need to check that $\Phi$ has a H\"older bound in $x$. Using Lemma \ref{use} (ii) and (iv) we have
\[
\begin{split}
\lefteqn{
\frac{|\Phi(N,x) - \Phi(N,y)|}{\| N\|+1}  } \\= {} & \frac{1}{\| N \| +1} \left| \bar{F} (S(x)+T(N,x),x) - \bar{F}(S(y)+T(N,y),y) \right| \\
\le {} &  \frac{1}{\| N \| +1}\bigg( \left| \bar{F} (S(x)+T(N,x),x) - \bar{F}(S(x)+T(N,x),y) \right|\\
&+\left| \bar{F} (S(x)+T(N,x),y) - \bar{F}(S(y)+T(N,y),y) \right|\bigg)\\
\le {} & K|x-y|^\beta+\frac{2n\Lambda}{ \| N \| +1} \| S(x)+T(N,x)  -S(y)  - T(N,y) \| \\
\le{} & K|x-y|^\beta+\frac{2n\Lambda \| S(x) -S(y) \| }{ \| N \| +1}  + \frac{2n \Lambda \| T(N,x) - T(N,y) \|}{\| N \| +1} \\
\le {} & K|x-y|^\beta+2K n\Lambda |x-y|^{\beta} + 2K n \Lambda |x-y|^\beta \le 5Kn\Lambda |x-y|^{\beta},
\end{split}
\]
as required.
\end{proof}

\section{Donaldson's equation} \label{sectiondon}

In this final section, we give an application of Theorem \ref{theoremCG} to an equation of Donaldson.

We begin by recalling some basic definitions.  A symplectic form $\omega$ on a manifold $M$ \emph{tames} an almost complex structure $J$ if at each point of $M$, $\omega(X, JX) >0$ for all nonzero vectors $X$. We can define a Riemannian metric by $g_{\omega}(X,Y)=\frac{1}{2}(\omega(X,JY)+\omega(Y,JX))$.  If, in addition, $\omega(JX, JY) = \omega(X,Y)$ for all $X, Y$ then we say that $\omega$ is \emph{compatible} with $J$.

We now describe a conjecture of Donaldson.  Let $M$ be a compact $2n$ dimensional (real) manifold with an
 almost complex structure $J$ and a symplectic form $\Omega$ taming $J$.  Let $\sigma$ be a smooth (positive) volume form on $M$.   Suppose $\tilde{\omega}$ is a symplectic form on $M$ compatible with $J$, satisfying  $[\tilde{\omega}]= [\Omega] \in H^2(M, \mathbb{R})$ and solving the Calabi-Yau equation
 \begin{equation} \label{cy}
 \tilde{\omega}^{n} = \sigma.
\end{equation}
Donaldson conjectured that, for $n=2$, there are $C^{\infty}$ a priori bounds on $\tilde{\omega}$ depending only on $\Omega$, $J$ and $\sigma$.   If this result were true, it would have important applications to symplectic topology \cite{Do2, LZ,LZ2,Ta}.  Some partial results towards this conjecture were given in \cite{We3, TWY}.  In particular, the conjecture holds, for any $n$, if the curvature of the canonical connection of $(g_{\Omega}, J)$
satisfies a certain positivity condition \cite{TWY}. The conjecture was solved in several special cases, such as the Kodaira-Thurston manifold \cite{kt} and more general $T^2$-bundles over $T^2$ \cite{BFV, FLSV} assuming $T^2$ symmetry. The interested reader can consult the survey \cite{TW25}.

The result we prove is as follows:

\begin{theorem}  \label{theoremdon}  Let $\tilde{\omega} \in [\Omega]$ solve (\ref{cy}) on $(M^{2n}, J, \Omega)$ with the notation as above.
 Fix $\alpha_0 \in (0,1)$.
Suppose that on a geodesic $g_{\Omega}$-ball $B_R$ of radius $R>0$ in $M$ we have
\begin{equation} \label{tr}
\emph{tr}_{g_{\Omega}}{\, \tilde{g}} \le C_0,
\end{equation}
where $\tilde{g}:= g_{\tilde{\omega}}$ is the metric associated to $\tilde{\omega}$.
Then there exist $\alpha \in (0,1)$ and $C>0$ depending only on $M, \Omega, J, R, \alpha_0$ and $\| \sigma \|_{C^{\alpha_0}(B_R, \, g_{\Omega})}$ such that
$$ \| \tilde{g} \|_{C^{\alpha}(B_{R/2}, \, g_{\Omega})} \le C.$$
\end{theorem}
In fact, one can easily derive from this higher order estimates $$ \| \tilde{g} \|_{C^{k}(B_{R/4}, \, g_{\Omega})} \le C_k,$$ for all $k\geq 1$, exactly as in \cite[Section 5]{TWY}, of course depending on higher derivatives of $\sigma$.
In the case when $\Omega$ is compatible with $J$ (not just taming), such a local estimate on the $C^{\alpha}$ norm of $\tilde{g}$ was proved by the third-named author \cite{We3} depending on two derivatives of $\sigma$, by adapting the Evans-Krylov method.  A global Calabi-type estimate on $\tilde{g}$ was proved by Tosatti-Weinkove-Yau \cite{TWY}, depending on three derivatives of $\sigma$.  Note that we state Theorem \ref{theoremdon} as a local rather than global result, because we anticipate that such local arguments may be useful in any future progress on  the full conjecture of Donaldson (cf. \cite{TW25}).

We now give the proof.

\begin{proof}[Proof of Theorem \ref{theoremdon}]  By shrinking $R$ if necessary, we may assume without loss of generality that $B_R$ is contained in a single coordinate patch for $M$, and that it is contractible.
Write $\Delta$ for the following Laplace operator associated to $g_{\Omega},$
$$\Delta u = \frac{n( \Omega^{(1,1)} )^{n-1} \wedge (dJdu)}{( \Omega^{(1,1)} )^{n}}=\tr{\Omega^{(1,1)}}{(dJdu)},$$
which in \cite{TWY} is called the canonical Laplacian of $g_{\Omega}$ (up to a factor of $2$). Note that
$$\tr{\Omega^{(1,1)}}{(\Omega)}=\tr{\Omega^{(1,1)}}{(\Omega^{(1,1)})}=n,$$
and
$$\tr{\Omega^{(1,1)}}{(\ti{\omega})}=\frac{1}{2}\tr{g_{\Omega}}{\tilde{g}}.$$
We let $u$ solve the Dirichlet problem
$$\Delta u =  \frac{1}{2}\tr{g_{\Omega}}{\tilde{g}} - n  \quad \textrm{on } B_R, \quad u|_{\partial B_R} =0.$$
From the assumption (\ref{tr}) and the equation (\ref{cy}), the metric $\tilde{g}$ is uniformly equivalent to $g_{\Omega}$.  By standard linear elliptic theory, $u$ is uniformly bounded in $W^{2,p}(B_R)$ for any given $p>1$. From now on, we fix a value of $p>n$.

Since $\tilde{\omega} -\Omega - dJdu$ is exact, there is a $1$-form $a_0$ on $B_R$ with $da_0=\tilde{\omega} - \Omega - dJdu$. We solve the Neumann problem
\begin{equation}\label{neum}
\Delta_d v=-d^*a_0\quad \textrm{on } B_R,\quad *dv|_{\de B_R}=-*a_0|_{\de B_R},
\end{equation}
where $*$ is the Riemannian Hodge star operator of $g_\Omega$, $d^*$ is the adjoint of $d$ with respect to $g_\Omega$, and $\Delta_d=d^*d$ is the Hodge Laplacian on functions.
As usual given a differential form $\alpha$ on $\ov{B_R}$, we write $\alpha|_{\de B_R}$ for the pullback of $\alpha$ under the inclusion of $\de B_R$ into $\ov{B_R}$. Then \eqref{neum} is indeed a Neumann problem, because
$$*dv|_{\de B_R}=\de_\nu v\  dS, \quad *a_0|_{\de B_R}=i_\nu a_0\  dS,$$
where $\nu$ is the $g_\Omega$-unit outward normal vector to $\de B_R$, $i_\nu$ is the interior product with $\nu$, and $dS$ is volume form on $\de B_R$ induced by $g_\Omega$.

We conclude that $a:=a_0+dv$ satisfies
$$\tilde{\omega} = \Omega + dJdu + da, \quad d^* a=0,\quad *a|_{\partial B_R}=0.$$
We can then apply for example \cite[Theorem D (b)]{We} to get
$$\|a\|_{W^{1,p}(B_R)}\leq C\|da\|_{L^p(B_R)}\leq C,$$
for a uniform constant $C$, thanks to assumption \eqref{tr} and the $W^{2,p}(B_R)$ bound for $u$.

Now observe that on $B_R$, $a$ solves the elliptic system (cf. \cite[equation (5.4)]{TWY})
$$da \wedge ( \Omega^{(1,1)} )^{n-1}  = 0, \quad d^* a=0, \quad (da)^{(2,0)+(0,2)} = - (\Omega+ dJdu)^{(2,0)+(0,2)}.$$
Equation \eqref{20} shows that $(dJdu)^{(2,0)+(0,2)}$ does not contain any second derivatives of $u$, and in fact it depends linearly on the gradient of $u$.
Since $u$ is bounded in $W^{2,p}$ it follows that $(dJdu)^{(2,0)+(0,2)}$ is bounded in $W^{1,p}$ and hence in $C^{\alpha}$ for some $\alpha \in (0,1)$.  By the standard interior elliptic estimates and the fact that $a$ is already bounded in $W^{1,p}$, we obtain $C^{1,\alpha}$ bounds for $a$  in a slightly smaller ball.  Hence $da$ is bounded in $C^{\alpha}$ in that ball.

Since $\tilde{\omega} = \Omega+da + dJdu$,
we can now write equation (\ref{cy}) locally as
$$(\chi + (dJd u)^{(1,1)})^n = \sigma,$$
for $\chi=\Omega+da+(dJdu)^{(2,0)+(0,2)}$ a form of type $(1,1)$ which is bounded in $C^{\alpha}$.  The theorem then follows immediately from the (local version of) Theorem \ref{theoremCG} for the almost complex version of (\ref{eqnma}).
\end{proof}

\end{document}